%% file: two-permnewarxiv.tex
\documentclass{amsart}

\usepackage{a4wide}

\usepackage{amssymb,amsmath,amsthm}%,txfonts}
\usepackage{amsmath, amssymb, epsfig}
\usepackage[active]{srcltx}
\usepackage{graphicx}
\usepackage{enumerate}
\usepackage{amsfonts}
\usepackage{eucal}
\usepackage{amsbsy}
\usepackage{verbatim}
\theoremstyle{plain} \newtheorem{theorem}{Theorem}

\newtheorem{lemma}[theorem]{Lemma}

\newtheorem{proposition}[theorem]{Proposition}

\newtheorem{claim}{Claim}

\newtheorem{remark1}[theorem]{Remark}
\newenvironment{remark}{\begin{remark1}\rm}{\end{remark1}}

\input{defpart.tex}

\input{definitions.tex}

\newcommand{\CSP}[1]{\mbox{\rm CSP$(#1)$}}

\newcommand{\NP}{\mbox{\bf NP}}
\newcommand{\NL}{\mbox{\bf NL}}

\newcommand{\algB}{\operatorname{{\mathcal B}}}
\newcommand{\algA}{\operatorname{{\mathcal A}}}

\newcommand{\algF}{\operatorname{{\mathcal F}}}

\def\zd{,\ldots,}

\newcommand{\Pol}{\operatorname{Pol}}

\begin{document}
%\markboth{C. Carvalho \& A. Krokhin}{On algebras with many symmetric operations}

\title{ON ALGEBRAS WITH MANY SYMMETRIC OPERATIONS}

%\catchline{}{}{}{}{}
%
%%%%%%%%%%%%%%%%%%%%%%%%%%%%%%%%%%%%%%%%%%%%%%%%%%%%%%%%%%%%%%%%%%%%

\author{Catarina Carvalho}
\address{University of Hertfordshire,
Hatfield, AL10  9AB, UK}
\email{c.carvalho2@herts.ac.uk}

\author{Andrei Krokhin}
\address{ Durham University,
Durham, DH1 3LE, UK}
\email{andrei.krokhin@durham.ac.uk}

\maketitle

\begin{abstract}
%Set functions, also known as totally symmetric operations, are well-known operations in
%universal algebra and play a role in the algebraic approach to the constraint satisfaction problem.
 An $n$-ary operation $f$  is called
 {\it symmetric} if,  for all permutations  $\pi$ of $\{1, \ldots, n\}$, it satisfies the identity $f(x_1, x_2, \ldots, x_n)=f(x_{\pi(1)}, x_{\pi(2)}, \ldots, x_{\pi(n)}).$ 
 We show that, for each finite algebra $\algA$, either it has symmetric term operations of all arities
or else some finite algebra in the variety generated by $\algA$ has two automorphisms without a common fixed point.
We also show this two-automorphism condition cannot be replaced by a single fixed-point-free automorphism.
%As a consequence, relational structures consisting of graphs of two permutations without a common fixed point are, in a sense, the forbidden structures for set functions.
\end{abstract}

\section{Introduction}

The study of algebras  with particular types of term operations has always been a subject of interest in the field of Universal Algebra~\cite{burrissanka,hobby-mck}, and it has boomed  since its connections with the complexity of Constraint Satisfaction Problems (CSPs) was discovered (see, e.g.~\cite{bulatovjeavonskrokhin,bulatovvaleriote}). Many complexity classification results for the CSP are based on algebraic dichotomies (of independent interest) of the form: either an algebra $\algA$ has term operations satisfying certain ``nice'' identities or else some finite algebra in the variety generated by $\algA$ has some ``bad'' compatible relational structure, often of a simple form (see~\cite{bartoabsorbing,barto:bw,bulatovjeavonskrokhin,bulatovvaleriote,Larose07:Universal}). Such structures are the forbidden structures for these ``nice'' term operations. We give several examples of such algebraic dichotomies in Section~\ref{sec3}.

In this paper, we identify the forbidden structures for having {\em symmetric} operations (of all arities) as term operations. A symmetric operation is an operation that is invariant under any permutation of arguments.
Such operations have recently been used in the algebraic approach to the CSP~\cite{kun},
they characterise the CSPs solvable by a natural algorithm based on linear programming. One intended use of our forbidden structures  is in proofs of computational hardness results, such as non-existence of certain robust algorithms \cite{dalmau:robust, kun}, for CSPs that cannot be solved by linear programming.

%\section{Relational structures, algebras and polymorphisms}
\section{Definitions}\label{pre}

The definitions given in this section are standard.
A {\em vocabulary} $\tau$ is a finite set of relation symbols $R_1, \ldots, R_k$ or arities $r_1, \ldots, r_k\ge 1$.
A $\vocab$-structure $\best$ consists of
a finite set $B$, called the {\em universe} of $\best$, and a relation
$R^{\best}\subseteq B^r$ for every relation symbol $R\in\vocab$
where $r$ is the arity of $R$.

A {\em homomorphism} from a $\vocab$-structure $\aest$ to a $\vocab$-structure
$\best$ is a mapping $h:A\rightarrow B$ such that for every $r$-ary $R\in\vocab$ and
every $(a_1,\dots,a_r)\in R^{\aest}$, we  have $(h(a_1),\dots,h(a_r))\in R^{\best}$.
We  write $\aest\rightarrow\best$ if there is a
homomorphism from $\aest$ to $\best$.

The {\em constraint satisfaction (or homomorphism) problem} for a structure $\best$ is testing whether
a given structure $\aest$ admits a homomorphism to $\best$. This problem is denoted by $\CSP\best$, and can be identified
with the class of all structures $\aest$ such that $\aest\rightarrow\best$.

Let $f$ be an $n$-ary operation on $B$, and $R$ a relation of $\best$. We say that $f$
is a {\it polymorphism} of $R$ if, for any tuples, $\bar{a}_1, \ldots, \bar{a}_n\in
R$, the tuple obtained by applying $f$ componentwise to $\bar{a}_1, \ldots,
\bar{a}_n$ also belongs to $R$. In this case, $R$ is said to be {\em invariant} under $f$, or {\em compatible} with $f$.
Furthermore, $f$ is a {\it polymorphism of $\mathbf{B}$} if it is a
polymorphism of each relation in  $\mathbf{B}$. It is easy to
check that the $n$-ary polymorphisms of $\best$ are precisely the
homomorphisms from the $n$-th direct power $\best^n$ to $\best$. We denote by $\Pol(\best)$ the set of all polymorphisms of $\mathbf{B}$.

A \emph{finite algebra} is a pair  $\algA=(A, F)$ where $A$ is a finite set and $F$ is a family of operations of finite arity on $A$. The {\em term operations} of $\algA$ are the operations obtained from $F$ and the projections by superposition.
A \emph{variety} is a class of (indexed) algebras closed under taking homomorphic images, subalgebras, and direct products. The
\emph{variety generated by} $\algA$, $\var( \algA)$, consists of all homomorphic images of subalgebras of direct powers of $\algA$. As usual, $HS(\algA)$ denotes the class of all homomorphic images of subalgebras of $\algA$.
From each relational structure $\best$, one can obtain an algebra $\algA_{\best}=(B, \Pol (\best))$, by taking as operations on the universe $B$ the polymorphism of all relations in $\best$. A structure $\best'$ with universe $A$ is said to be {\em compatible} with an algebra $\algA=(A,F)$ if every operation in $F$ is a polymorphism of $\best'$ (equivalently, each relation in $\best'$ is the universe of a subalgebra of the corresponding power of $\algA$).

%We can then speak of the CSP of an algebra, $\CSP{\algA}$, and we will use both terms interchangeably in this paper.

The notion of a polymorphism plays the key role in the algebraic
approach to the CSP. The polymorphisms of a structure are
known to determine the complexity of $\CSP\best$ as well as
definability of (the complement of) $\CSP\best$ in various logics (see~\cite{surveydualities,Larose07:Universal}).

We now define several types of operations that will be used in this paper.

\begin{itemize}
\item
An $n$-ary operation $f$  is called {\it idempotent} if it satisfies the
identity $f(x, \ldots, x)=x$.
\item An $n$-ary operation $f$  is called  {\it cyclic} if it satisfies the identity $$f(x_1, x_2, \ldots, x_n)=f(x_2, x_3,\ldots, x_n, x_1);$$
\item An $n$-ary operation $f$  is called
 {\it symmetric} if  it satisfies the identity $$f(a_1, a_2, \ldots, a_n)=f(a_{\pi(1)}, a_{\pi(2)}, \ldots, a_{\pi(n)})$$ for all permutations $\pi$ of $\{1, \ldots, n\}$;
\item
An $n$-ary operation $f$ is called {\it totally symmetric} if $f(x_1, \ldots,
x_n)=f(y_1, \ldots, y_n)$ whenever $\{x_1, \ldots, x_n\}=\{y_1, \ldots, y_n\}$.
If, in addition, $f$ is idempotent then we say that it is a TSI operation.

\item An $n$-ary ($n\ge 3$) operation is called a {\em weak near-unanimity (WNU)}
operation if it is idempotent and it satisfies the identities \[f(y,x\zd x,x)=f(x,y\zd x,x)=\ldots
=f(x,x\zd x,y).\]

%\item An $n$-ary ($n\ge 3$) operation is called a {\em near-unanimity (NU)}
%operation if it satisfies the identities \[f(y,x\zd x,x)=f(x,y\zd x,x)=\ldots
%=f(x,x\zd x,y)=x.\]
%
%\item A ternary NU operation is called a {\em majority} operation.

\item A {\em Mal'tsev} operation is a ternary operation $f$ satisfying
\[f(x,x,y)=f(y,x,x)=y.\]
\end{itemize}

More of the universal-algebraic background can be found in~\cite{burrissanka,hobby-mck}.

\section{Some algebraic dichotomies}\label{sec3}

We will now describe some known algebraic dichotomy results and indicate where they are used in the study of CSPs.
It is known~\cite{bulatovjeavonskrokhin} that it is enough to classify only problems $\CSP\best$ such that the corresponding algebra $\algA_{\best}$ is idempotent,
i.e. all of its operations are idempotent. This explains why most of the dichotomies concern only idempotent algebras.

\begin{enumerate}
\item[(i)]  For a finite idempotent algebra $\algA$, either $\algA$ has a cyclic operation of some arity (equivalently, $\var( \algA)$ satisfies a non-trivial Mal'tsev condition), or else
the ternary relation $\{(0,0,1),(0,1,0),(1,0,0)\}$ is compatible with some (2-element) algebra in $HS(\algA)$~\cite{bartoabsorbing}.

It is known that, for a structure $\best$, if the (idempotent) algebra $\algA_{\best}$ satisfies the latter
condition then $\CSP\best$ is $\NP$-complete~\cite{bulatovjeavonskrokhin}. The Algebraic Dichotomy Conjecture
states that if $\algA_{\best}$ satisfies the former condition then $\CSP\best$ is tractable~\cite{bartoabsorbing,bulatovjeavonskrokhin}.

\item[(ii)] For a finite idempotent algebra $\algA$, either $\algA$ has WNU operations of  almost all arities (equivalently, $\var( \algA)$ is congruence meet-semidistributive), or else
there exists an algebra $\algB$ in $HS(\algA)$ and an Abelian group structure on the base set of $\algB$  such that the relation $\{(x, y, z):x+y=z\}$ is compatible with $\algB$~\cite{Larose07:Universal,marotimckenzie}.

It is known that the former condition, for the algebra $\algA_{\best}$, implies that $\CSP\best$ is definable
in the logic programming language Datalog~\cite{barto:bw} (and also admits a robust algorithm~\cite{barto:robust}), while the latter condition, which intuitively
says that $\CSP\best$ can encode systems of linear equations, implies the absence of these nice properties~\cite{federvardi,dalmau:robust}.

\item[(iii)]  For a finite idempotent algebra $\algA$, either $\algA$ has ternary term operations from Theorem~9.11 of~\cite{hobby-mck} (equivalently, $\var( \algA)$ is congruence join-semidistributive), or else
    there exists an algebra $\algB$ in $HS(\algA)$ such that at least one of the relations $\{(x, y, z):x+y=z\}$ (as above) and $\{0,1\}^3\setminus \{(1,1,0)\}$ is compatible with $\algB$~\cite{Larose07:Universal}.

The former condition, for the algebra $\algA_{\best}$, is conjectured to imply that $\CSP\best$ is definable
in linear Datalog~\cite{Larose07:Universal} (which roughly means that $\CSP\best$ can be reduced to the Digraph Reachability problem) and belongs to the complexity class $\NL$, while the latter condition, which intuitively
says that $\CSP\best$ can encode systems of linear equations or Horn 3-Sat, implies non-definability in linear Datalog and non-membership in $\NL$ (modulo complexity-theoretic assumptions)~\cite{Larose07:Universal}.

\item[(iv)] For a finite idempotent algebra $\algA$, either $\algA$ has a Mal'tsev operation as a term operation (equivalently, $\var( \algA)$ is congruence permutable), or else some binary reflexive and non-symmetric relation
    is compatible with a finite algebra $\mathcal{V}(\algA)$~\cite{hagemannmitschke}.

The latter condition was used in~\cite{bulatovcounting} to prove hardness of the counting version of $\CSP\best$, and in~\cite{bulatovmarx} to prove hardness of a version of $\CSP\best$ with an additional global constraint.
\end{enumerate}

\section{Forbidden structures for many symmetric operations}

Since the presence of many symmetric operations plays a role in the study of CSPs, it is natural to try to find (simple enough)
forbidden structures for this algebraic condition.

%Recently,  Kun et al. showed in ~\cite{kun} that a structure admits a set function
%if and only if it has symmetric polymorphisms of all arities.

For a permutation $\pi$ on $A$, let $\pi^\circ$ denote the graph of $\pi$, i.e. $\pi^\circ=\{(a,\pi(a))\mid a\in A\}$.
In the next two sections we will deal with graphs of permutations compatible with algebras.
Note that $\pi^\circ$ is compatible with an algebra $\algA$ if and only if $\pi$ is an automorphism of $\algA$.

%It is clear that if the graph of a cyclic permutation, i.e. a relation of the form
%$$R=\{(a_1, a_2), (a_2, a_3), \ldots, (a_{n_1}, a_n), (a_n, a_1)\}$$
%with the $a_i$s all different is compatible with an algebra $\algA$, then $\algA$ cannot have an $n$-ary symmetric term %operation. In fact it cannot even have a cyclic operation of arity $n$ as a term operation, for  if $f$ is a $n$-ary cyclic %operation we have
%$$f((a_1, a_2), (a_2, a_3), \ldots, (a_{n_1}, a_n), (a_n, a_1))=(f(a_1, \ldots, a_n), f(a_2, \ldots, a_n, a_1))=(a, a)$$
%and clearly no pair of the form $(a, a)$ is in $R$.

The following is a slightly weakened Proposition 2.1 of~\cite{bartocyclic}.

\begin{proposition}\label{allcyclic}
Let $\algA$ be a finite algebra.
\begin{itemize}
\item Either $\algA$ has cyclic term operations of all arities,
\item or else
 there is a finite algebra $\algB$ in $\var(\algA)$ with a fixed-point-free automorphism.
\end{itemize}
\end{proposition}

Since any symmetric operation is cyclic, the latter condition in Proposition~\ref{allcyclic} is sufficient to forbid the existence of symmetric term operations of all arities. Could it also be necessary, at least for algebras of the form $\algA_{\best}$? We will show that it is not, but a small variation of it is such a condition.

\begin{theorem}\label{2cycles}
Let $\algA$ be a finite algebra.
\begin{itemize}
\item Either $\algA$ has symmetric term operations of all arities,
\item or else
 there is a finite algebra $\algB$ in $\var(\algA)$ that has two automorphisms without a common fixed point.
  Furthermore,  one of the automorphisms can be chosen to have order two.
\end{itemize}
\end{theorem}
\begin{proof}
It is easy to see that if $f$ is an $n$-ary symmetric term operation of $\algA$, and hence of every algebra in $\var( \algA)$, then, for any algebra $\algB$ in $\var( \algA)$ with universe $\{b_1,\ldots,b_n\}$, the element $f(b_1,\ldots,b_n)$ is a fixed point of every automorphism of $\algB$.

Assume now that $\algA$ does not have a symmetric operation of arity $n$.
Let $\algF$ be the free $n$-generated algebra in the variety $\var( \algA)$, with free generators $x_1, x_2, \ldots, x_n$. Let $\algA_1$ and $\algA_2$ be the subalgebras of $\algF\times \algF$ generated by the tuples
$$\algA_1=\langle \  (x_1, x_2), (x_2, x_1), (x_3, x_3), \ldots, (x_n, x_n)\ \rangle$$
$$\algA_2=\langle \ (x_1, x_2), (x_2, x_3), (x_3, x_4), \ldots, (x_{n-1}, x_n), (x_n, x_1) \ \rangle.$$
Since $x_1, \ldots, x_n$ are the free generators of $\algF$, the universes of $\algA_1$ and $\algA_2$ can be thought of as graphs of permutations on the universe of $\algF$ (and hence correspond to automorphisms of $\algF$).
The automorphism corresponding to $\algA_1$ has order two.
If these permutations share a fixed point then
there exist $n$-ary operations $f_1$ and $f_2$ and an element $a$ in $\algF$ such that
$$\begin{array}{l}f_1((x_1, x_2),(x_2, x_1), (x_3, x_3), \ldots, (x_n, x_n))= \\
f_2((x_1, x_2), (x_2, x_3),  \ldots, (x_{n-1}, x_n), (x_n, x_1) ) =(a,a).\end{array}$$
This implies that $f_1=f_2$ and, moreover, $f_1(x_1, x_2, x_3, \ldots, x_n)=f_1(x_2, x_1, x_3, x_4, \ldots, x_n)=f_1(x_2, x_3, \ldots, x_n, x_1)$, and so $f_1$ is symmetric. Hence we have an $n$-ary symmetric operation in $\algA$, a contradiction.
\end{proof}

The classes of algebras appearing in Proposition~\ref{allcyclic} and Theorem~\ref{2cycles} are different, as our next result shows.
Hence, graphs of fixed-point-free permutations do not form a complete set of forbidden structures for the existence of symmetric term operations of all arities.

Let $\kest=(K; R, S)$ be the structure with domain
 $$K=\{0, 1,2, \ldots,9, 10, \overline{01}, \overline{02}, \overline{03}, \overline{04}, \overline{12}, \overline{13}, \overline{14}, \overline{23}, \overline{24}, \overline{34}\},$$
 and binary relations $R$ and $S$ that are graphs of the following permutations $r$ and $s$, respectively,
$$r=(0 \ 1\  2) (5 \ 6 \ 7)(8\  9\  10) (\overline{12}\  \overline{02}\  \overline{01})(\overline{04}\ \overline{14}\ \overline{24})(\overline{13}\  \overline{23}\ \overline{ 03}),$$
$$s=(1 \ 4)(2\  3) (5\  6)(7\  8) (\overline{34}\ \overline{12})(\overline{02}\  \overline{03})(\overline{01}\  \overline{04})(\overline{24}\ \overline{13}).$$
It will be often convenient to think of $\kest$ as of two graphs as depicted in Figure~\ref{graph}, one directed, $R$ (represented by the solid lines), and one undirected, $S$ (represented by the dotted lines),  on the same set of vertices $K$. Then fixed points of permutations correspond to loops in the graphs.
To simplify notation, for elements of the form  $\overline{xy}$ we assume the convention that $\overline{xy}=\overline{yx}$.

\begin{figure}[h]
 \centering
    \includegraphics[scale=0.65]{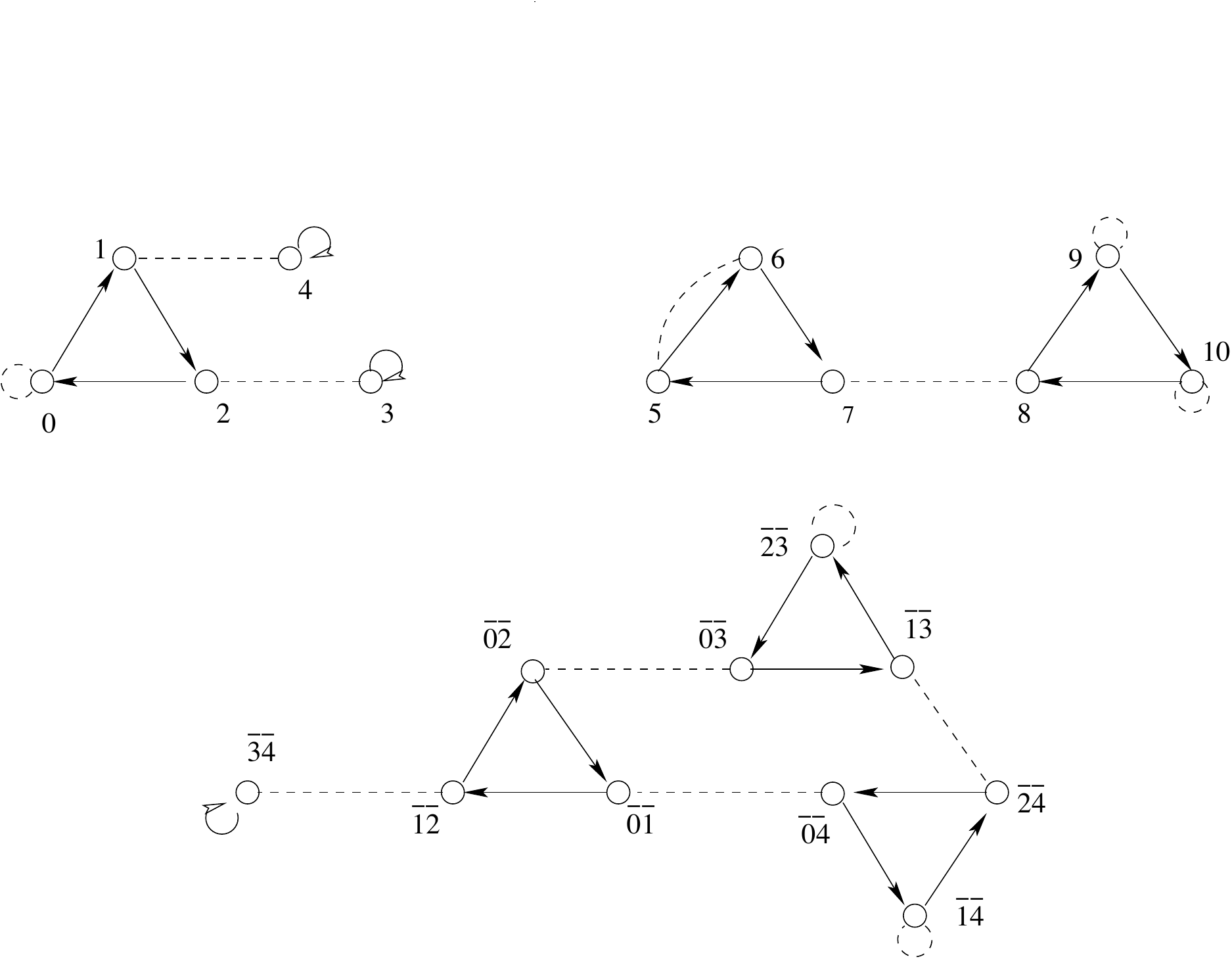}
     \caption{Relational structure $\kest$}
 \label{graph}
\end{figure}

\bigskip

\begin{theorem}\label{cic}
The structure $\kest$ has cyclic polymorphisms of all arities, but no symmetric polymorphism of arity 5.
\end{theorem}

\begin{proof}
 For a contradiction, suppose that $f$ is a 5-ary symmetric operation that preserves both $R$ and $S$. Since $f$ is symmetric we know that $f(0,1,2,3,4)=f(1,2,0,3,4)=f(0,4,3,2,1)$. We have that  $(0,1),(1,2),(2,0),(3,3),(4,4)\in R$ and $(0,0), (1,4),(2,3),(3,2),(4,1)\in S$. It follows that $f(0,1,2,3,4)$ is a common loop in $R$ and $S$, which does not exist, a contradiction.
The proof that $\kest$ has cyclic polymorphisms of all arities occupies the next section.
\end{proof}

We used 
computer-assisted search in the process of finding the structure $K$, but 
the search was not designed to find the smallest structure with required 
properties, so we do not know whether $K$ is smallest.
We do know that this example of structure  is tight in the sense that the existence of  cyclic operations of all arities imply the existence of symmetric operations of arities up to 4. 

\begin{lemma}\label{s4}
If an algebra $\algA$ has cyclic  term operations of arities 2 and 3 then it also has   symmetric term operations of arities up to  4.
\end{lemma}

\begin{proof}
Let $s_2, c_3$ be cyclic operations of arities 2 and 3 respectively. Clearly $s_2$ is symmetric, and it is easy to check that the operation
$$s_3(x,y,z)=s_2(c_3(x,y,z), c_3(y,x,z))$$
is a 3-ary symmetric operation.

%Let $s_2$ and $s_3$ be symmetric operations of arities 2 and 3 respectively.
Now, the 4-ary operation
$t(x,y,z,w)=s_2(s_2(x,y),s_2(z,w))$ satisfies the following identities
$$\begin{array}{l}
t(x,y,z,w)=t(y,x,z,w)=t(x, y, w, z)=t(y, x, w, z)\\
= t(z,w,x,y)=t(z,w,y,x)=t(w,z,x,y)=t(w,z,y,x).\end{array}$$
It then follows that the operation
$$s_4(x,y,z,w)= s_3(t(x,y,z,w), t(x,w,y,z), t(x,z,y,w))$$
is  symmetric.
\end{proof}

\begin{remark}
The condition of having totally symmetric operations of all arities has also played a role in the study of the CSP.
Such operations characterise the so-called CSPs of width 1, i.e. CSPs solvable by the arc-consistency algorithm~\cite{dalmauwidth1,federvardi}.
It was claimed in~\cite{kun} that this condition is equivalent to the one of having many symmetric operations, but a flaw was discovered in the proof (as acknowledged on R.~O'Donnell's webpage), and a counter-example to the claim was recently found by G.~Kun~\cite[Example 99]{kun-example}: a very simple structure that
has symmetric polymorphisms of all arities, but no ternary totally symmetric polymorphism.
\end{remark}

\section{Proof of Theorem~\ref{cic}}

We make use of  two results that have been proved for algebras, that  can naturally be applied to relational structures.

\begin{proposition}\cite[Proposition 2.2]{bartocyclic}
For a finite algebra $\mathcal{A}$ the following hold:
\begin{enumerate}
\item If  $\mathcal{A}$ has an $n$-ary cyclic term then it has a $k$-ary cyclic term for all $k>1$ divisor of $n$.
\item If $\mathcal{A}$ has an $n$-ary and an $m$-ary cyclic term, then it also has an $mn$-ary cyclic term.
\end{enumerate}\label{cyclicprime}
\end{proposition}

\begin{proposition}\cite[Theorem  4.1]{bartoabsorbing} \label{cyclicbigarities}
Let $\algA$ be a finite algebra. The following are equivalent
\begin{itemize}
\item $\algA$ has a cyclic term;
\item $\algA$  has a cyclic term of arity $p$, for every prime $p > |A|$.
\end{itemize}
\end{proposition}

It follows from Propositions~\ref{cyclicprime} and \ref{cyclicbigarities} that it is enough to show that $\kest$ has cyclic polymorphisms of arity $p$ for all prime $p<21$.
We will define partial cyclic operations of prime arities on $K$, and then show by induction that $\kest$ is preserved by cyclic operations of arities up to 21.

Consider the following partition of $K$: $C_1=\{0, \ldots, 4\}$, $C_2=\{5, \ldots10\}$,
  %which we rename to make it simpler to the define the operations on $V$, respectively,  as
%$$(01243), (02134),(01432),(03214),(02413),(01324)$$
and $C_3=\{\overline{01}, \ldots, \overline{34}\}$; blocks $C_1$ and $C_3$  are  depicted in Fig.~\ref{c2onC2} and  Fig.~\ref{c3onC2}, respectively. In these figures the filled lines represent the arcs of $R$ and the dotted lines the (undirected) edges of $S$.

We start by defining partial  cyclic  operations $c_p(x_1,\ldots, x_p)$ for all prime $p<21$ and $x_1, \ldots, x_p$ all belonging to the same block. These operations do not necessarily preserve the blocks but  preserve $R$ and $S$. Then, using these operations,  we show, by induction on $n$, that $R$ and $S$ are preserved by cyclic  operations of arity $n$, for all $n=2, \ldots, 21$. Recall that, for elements of the form  $\overline{xy}$, we assume the convention that $\overline{xy}=\overline{yx}$.

%It immediately follows that $\kest$ is  preserved by cyclic operations of all arities.

%of prime arities on each component (not necessarily preserving the component but preserving the edges) and then show by induction on $n$ that $r, s$ are preserved by a cyclic operation $c_n$ for all $n\ge2$.

%that we rename, respectively,
%$$\{3,4\},\{1,2\},\{0,2\},\{0,1\},\{0,4\},\{1,4\},\{2,4\},\{1,3\},\{2,3\},\{0,3\}.$$
\begin{enumerate}
\item[(1)] {\it{Definition of $c_p(x_1, \ldots, x_p)$ with $x_1, \ldots, x_p$ all distinct and belonging to the same block:}}\\
We assume that $x_1, \ldots, x_p$ are all distinct, so when they belong to $C_1$ or $C_2$ we just need to define $c_p$ for $p\le 5$, and when they belong to $C_3$ we define $c_p$ for $p\le 7$.

%We now define the cyclic idempotent operation $c_p$ on distinct elements  $x_1, \ldots, x_p$, belonging to the same component $C_i$, $i=1,2,3$.

We define the operation $c_2$ to act symmetrically on all blocks, i.e once we define it  on a tuple $(x,y)$ the definition is the same on the tuple $(y,x)$. 
 For distinct $x, y \in C_1$ we let $c_2(x, y)=\overline{xy}$; for distinct $x, y\in C_3$ we define $$c_2(x,y)=\left\{ \begin{array}{cl}
a& {\rm if}\ x=\overline{ab}, y=\overline{ac} \ { \rm for\  some} \ a\in C_1\\
e& {\rm if }\ x=\overline{ab}, y=\overline{cd}, \ { \rm and} \ C_1=\{a,b,c,d, e\};
\end{array}\right.$$
and for distinct $x,y \in C_2$ we define it as shown in Fig.~\ref{c2onC2}: $c_2(5,6)=c(7,8)=c_2(9,10)=0, \ c_2(5,7)=c_2(6,10)=c_2(8,9)=2$ and so on.
\begin{figure}[h]
 \centering
    \includegraphics[scale=0.65]{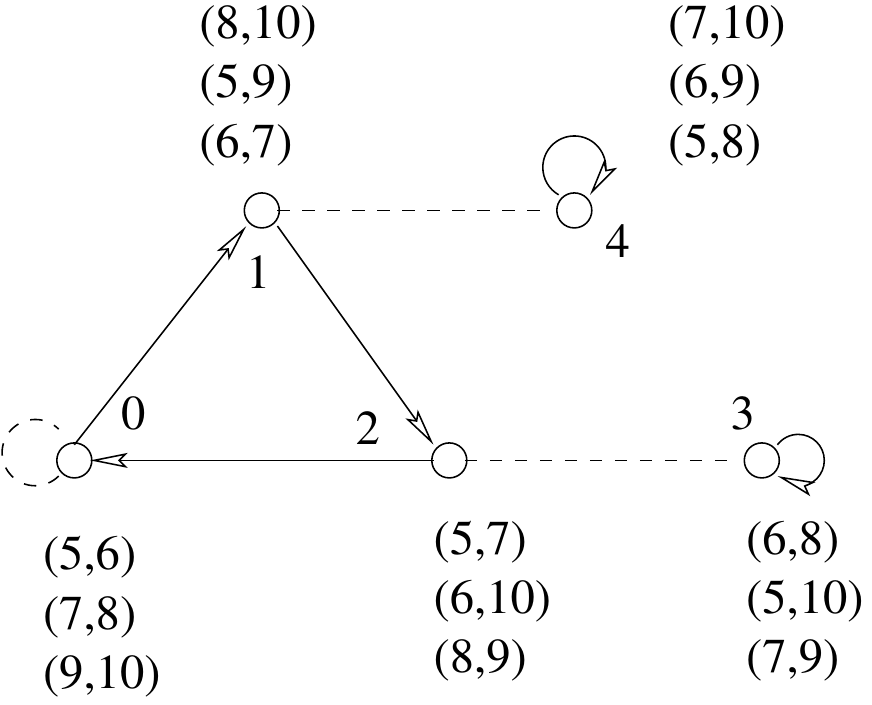}
     \caption{Operation $c_2$ maps  $C_2$ to $C_1$}
 \label{c2onC2}
\end{figure}

It is worth noting that operation $c_2$ and below  $c_3$ and $c_5$  are defined on the block $C_3$  by going  over all isomorphism 
types of graphs on $\{0,1,2,3,4\}$ that have $2,3$ and $5$ edges, respectively.

To check that $c_2$, as defined, preserves relations $R$ and $S$ we can start with any two elements of $C_1$ and follow their images in $C_3$, moving along elements connected first by $R$ and then  by $S$. Indeed component $C_3$ appeared by defining a binary symmetric relation on $C_1$. For elements of $C_2$ we can check that adjacency in both $R$ and $S$ is preserved  by $c_2$ by checking all mapping in Figure ~\ref{c2onC2}. Operation $c_2$ sends any two elements of $C_3$ to an element of $C_3$, we can also follow adjacency in block $C_3$ in Figure~\ref{graph};  e.g.  we have  for example $(\overline{01}, \overline{12}), (\overline{14}, \overline{24}) \in R$ and  $c_2( \overline{01}, \overline{14})= 1, c_2(\overline{12}, \overline{24})= 2$ and we can see that  $(1, 2)\in R$.

We define operation $c_3$  to also act symmetrically on all elements, i.e once we define it  on a tuple $(x,y,z)$  the operation takes the same value on any tuple  obtained by arbitrarily permuting   $x, y, z$. For distinct $x, y, z\in C_1$ we define $c_3(x, y, z)=c_2(u,v)$ where  $\{u, v\}=C_1\backslash \{x, y, z\}$;
when  $x, y,z\in C_3$ are all distinct, we let
$$c_3(x,y,z)=\left\{ \begin{array}{cl}
\overline{de} & {\rm if} \  \ x=\overline{ab},\  y=\overline{bc},\  z=\overline{ac},  \\
\overline{ae}  & {\rm if} \ \ x=\overline{ab},\  y=\overline{ac}, \ z=\overline{ad}, \\
a & {\rm if} \ \ x=\overline{ab},\  y=\overline{ac},\  z=\overline{de}, \\
e  & {\rm if}\  \ x=\overline{ab},\  y=\overline{bc},\  z=\overline{ad};
\end{array}\right.$$
and in $C_2$ we define $c_3$ as shown in Fig.~\ref{c3onC2}: $c_3(5,6,7)=c_3(8,9,10)=\overline{34}$, and so on.

\begin{figure}[h!]
 \centering
    \includegraphics[scale=0.65]{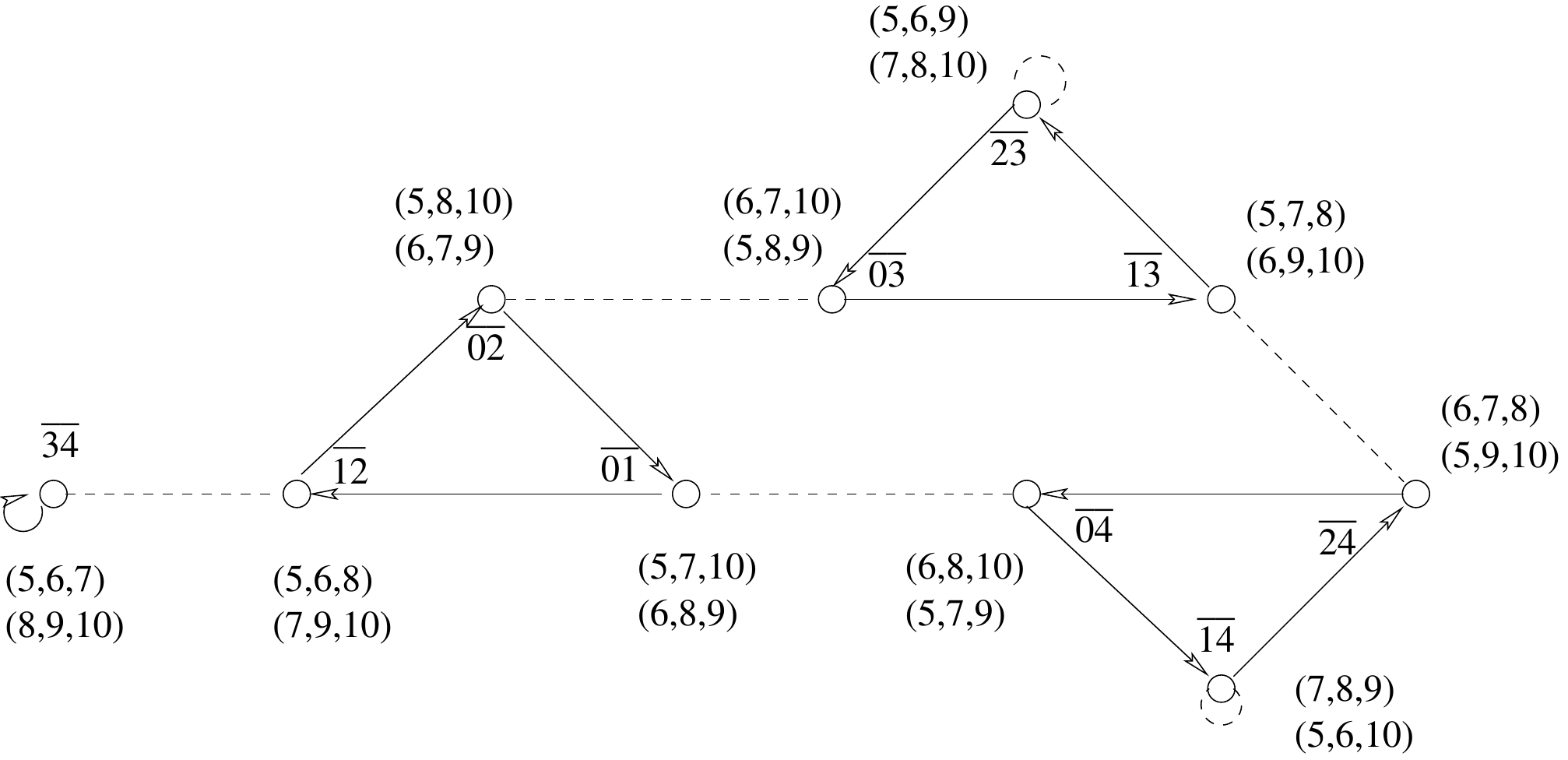}
     \caption{Operation $c_3$ maps $C_2$ to $C_3$}
 \label{c3onC2}
\end{figure}

To check that $c_3$, as defined, preserves relations $R$ and $S$  for any two elements of $C_1$ follows immediately from the fact that  $c_2$ preserves $R$ and $S$, for any two elements of $C_2$ we can follow the image of the elements in Fig.~\ref{c3onC2}; if we consider $c_3$ acting on elements of $C_3$ the idea is that triples of elements considered in each of the 4 cases of $c_3$ are connected, via $R$ and $S$, to triples considered in the same case type, and we can  then check that the relations are preserved by following the pictures of components $C_3$ and $C_1$, e.g. we have 
$$\begin{array}{l}  (\overline{03},\overline{13}), ( \overline{13}, \overline{23}), (\overline{01}, \overline{12})\in R \ {\rm and} \ (c_3( \overline{03}, \overline{13}, \overline{01}), c_3(\overline{13}, \overline{23}, \overline{12}))=(\overline{24}, \overline{04})\in R;\\
  (\overline{01},\overline{12}), ( \overline{02}, \overline{01}), (\overline{03}, \overline{13})\in R  \ {\rm and} \  (c_3( \overline{01}, \overline{02}, \overline{03}), c_3(\overline{12}, \overline{01}, \overline{13}))=(\overline{04}, \overline{14})\in R\\
  
    (\overline{01},\overline{04}), ( \overline{02}, \overline{03}), (\overline{34}, \overline{12})\in S  \ {\rm and} \  (c_3( \overline{01}, \overline{02}, \overline{34}), c_3(\overline{04}, \overline{03}, \overline{12}))=(0, 0)\in S\\
     (\overline{01},\overline{04}), ( \overline{12}, \overline{34}), (\overline{03}, \overline{02})\in S  \ {\rm and} \  (c_3( \overline{01}, \overline{12}, \overline{03}), c_3(\overline{04}, \overline{34}, \overline{02}))=(4, 1)\in S.
     \end{array}$$

We define the operation $c_5$ to be  cyclic in $C_1$ and symmetric on the remaining blocks, i.e once we define an operation on a tuple $(x_1, \ldots, x_5)$ it takes the same value 
on all tuples obtained from cycling permuting $x_1, \ldots, x_5$, and if $x_1, \ldots, x_5$ belong to $C_2$ or $C_3$  the operation also takes the same value on the  tuples obtained by arbitrarily permuting $x_1, \ldots, x_5$.  
When  $x, y, z,  u,v  \in C_2$ are all distinct, we define $c_5(x, y, z, u,v)=w$  where $C_2=\{x,y,z,u,v,w\}$; and for distinct  $x, y, z, u, v \in C_3$  we define
$$c_5(x,y,z,u,v)=\left\{ \begin{array}{cl}
a& {\rm if} \ x=\overline{ab},  y=\overline{ac},  z=\overline{ad},  u=\overline{ae},  v=\overline{ce}\\
a& {\rm if} \  x=\overline{ab},  y=\overline{cd},  z=\overline{eb},  u=\overline{bd},  v=\overline{ad}\\
e& {\rm if} \  x=\overline{ab},  y=\overline{cd},  z=\overline{cb},  u=\overline{bd},  v=\overline{ad}\\
e& {\rm if} \  x=\overline{ab},  y=\overline{cd},  z=\overline{cb},  u=\overline{bd},  v=\overline{ae}\\
c_5(a,b,c,d,e) &  {\rm if}  \ x=\overline{ab},  y=\overline{bc},  z=\overline{cd},  u=\overline{de},  v=\overline{ae}
\end{array}\right.$$
where $C_1=\{a, b, c, d, e\}$;  in $C_1$ we define
$$\begin{array}{c}
c_5(0,1,2,4,3)=5, \ \ \ c_5(0,4,3,1,2)=6, \ \ \ c_5(0,1,4,3,2)=7, \\ c_5(0,4,1,2,3)=8, \ \ \ c_5(0,2,4,1,3)=9, \ \ \ c_5(0,1,3,2,4)=10,\end{array}$$
and to extend  $c_5$ to the rest of  $C_1$, we think of the tuple  $(x,y,z,u,v)$ as the  permutation $(x  y z u v)$. It is then easy to see that $(xyzuv)$ is an $i^{th}$ power, with $i=0,1,\ldots,4$, of exactly one of six permutations corresponding to the tuples for which $c_5(x, y, z, u, v)$ was defined above. We then define $c_5(x, y, z, u, v)$ to be the same as $c_5$ applied to the corresponding  tuple, e.g. $(0 2 3 1 4)=(01243)^2$ and so $c_5(0,2,3,1,4)=5$.

% \begin{itemize}
%\item $c_5(0,1,2,4,3)=c_5(0,3,4,2,1)=c_5(0,4,1,3,2)=c_5(0,2,3,1,4)= 5,$
%\item $ c_5(0,4,3,1,2)=c_5(0,2,1,3,4)=c_5(0,1,4,2,3)=c_5(0,3,2,4,1)= 6,$
%\item  $c_5(0,1,4,3,2)=c_5(0,2,3,4,1)=c_5(0,3,1,2,4)=c_5(0,4,2,1,3)= 7,$
%\item  $c_5(0,4,1,2,3)=c_5(0,3,2,1,4)=c_5(0,2,4,3,1)=c_5(0,1,3,4,2)= 8,$
%\item $c_5(0,2,4,1,3)=c_5(0,3,1,4,2)=c_5(0,4,3,2,1)=c_5(0,1,2,3,4)=9,$
%\item  $c_5(0,1,3,2,4)=c_5(0,4,2,3,1)=c_5(0,2,1,4,3)=c_5(0,3,4,1,2)= 10.$
%\end{itemize}

%Given $x, y, z,  u,v  \in C_1$, then thinking of the tuple $(xyzuv)$ as a permutation, it is easy to check that there exists a power of the permutation that equals an element of $C_2$, i.e. $(xyzuv)^i\in C_2$ for some $i=0,\ldots, 4$, we then define  $c_5(x, y, z, u, v)= (xyzuv)^i $;

Finally, for distinct  $x_1, \ldots, x_7 \in C_3$
we define $ c_7(x_1, \ldots, x_7) = c_3(a,b,c)$, where  $\{a, b, c\}=C_3\backslash \{x_1, \ldots, x_7\}$.

%We will later show that when these vertices are all distinct  we can define an idempotent cyclic operation $c_p(x_1, \ldots, x_p)$. Furthermore, this operation is symmetric for all $p\neq 5$, and as a consequence  of Lemma~\ref{s4} we also have a symmetric idempotent operation or arity $4$, $c_4$, defined on distinct vertices belonging to the same component.

We now extend these operations to elements $x_1, \ldots, x_p$ belonging to the same component but not necessarily distinct.

\item[(2)] {\it {Definition of $c_p(x_1, \ldots, x_p)$ with $x_1, \ldots, x_p$ not all distinct and belonging to the same block, and $p$ any prime number:}}\\
For convenience, we define $c_1(x)=x$ for all $x\in V$.

\begin{claim}\label{4elements}
Let $p$ be any prime number, and $x_1, \ldots, x_p$ be elements from the same block of $K$.
 If  $|\{x_1, \ldots, x_p\}|\ge 5$ then there exists $k=1, \ldots, p$ such that at most 4 elements of $\{x_1, \ldots, x_p\}$  appear exactly $k$ times in $x_1, \ldots, x_p$.
    \end{claim}
  \begin{proof}
  For all $i=1, \ldots, p$,
   let $N_i$ be the (possibly empty) set of elements that appear exactly $i$ times in  $x_1, \ldots, x_p$.
   Note that $|N_1|< p$, since $x_1, \ldots, x_p$ are not all distinct, and $|N_p|=0$ because $|\{x_1, \ldots, x_p\}|\ge 5$.
  %Since the biggest component, $C_3$, has 10 elements we know that $|N_i|\le 10$.
  We have $p=\sum_{i=1}^{p}i|N_i|$, which implies that there are at least two $i$'s for which $N_i$ is non-empty, for $p$ is prime. Let $j_1$ and $j_2$ be the smallest and largest $i$, respectively,  for which   $N_i$ is non-empty. We show that at least one of  $N_{j_1}$,    $N_{j_2}$ has at most 4 elements.   Suppose, for a contradiction, that $|N_{j_1}|\ge 5$ and  $|N_{j_2}|\ge 5$. Then the $x_i$'s must all come from $C_3$. The set $N_{j_1}\cup N_{j_2}$ contains 10 different elements, i.e. all elements of $C_3$. It follows all the other sets $N_i$ are empty,  and $p=5j_1+5j_2$, a contradiction.
  \end{proof}
We then define $c_p(x_1, \ldots, x_p)=c_j(y_1, \ldots, y_j)$, where $j\le 4$, and either $\{x_1, \ldots, x_p\}=\{y_1, \ldots,  y_j\}$  or, when $|\{x_1, \ldots, x_p\}|\ge 5$,  $y_1, \ldots, y_j$ are the (at most 4)  elements repeated exactly $k$ times mentioned in Claim~\ref{4elements} (and $k$ is the smallest such value). Note that the elements $y_1, \ldots, y_j$ are all 
distinct and come from the same block, so $c_j(y_1,\ldots ,y_j)$ is already 
defined.
%When at most 4  elements exist for more than one value of $k$, we let $y_1, \ldots, y_j$ be the elements given by the smallest value of $k$.
Indeed,  $c_j(y_1, \ldots, y_j)$ is defined to be symmetric in (1) for $j\le 3$. Also by (1) and using (the proof of) Lemma~\ref{s4} we know that there exists a symmetric operation $s_4$ defined on elements $x_1, \ldots, x_4$ all distinct and belonging to the same block.
It follows that $c_p$ acts on $x_1, \ldots, x_p$ as a symmetric operation.

%We start by defining the  cyclic operations $c_p(x_1, \ldots, x_p)$ for $x_1, \ldots, x_p$ all belonging to the same component, for $p$ a prime less than $21$. We will show later that whenever $x_1, \ldots, x_p$ are distinct elements we can define these operations to be symmetric for $p=2,3$. It then immediately follows, by Lemma(????) that  we also have a symmetric operation of arity 4, $c_4$. We will use this $4$-ary operation to define operations of higher arities.

%We assume that for all $p$, $c_p(x_1, x_2, \ldots, x_p)$ is defined when $|\{x_1, \ldots, x_p\}|=p$, and  define  $c_p(x_1, \ldots, x_p)$ for $|\{x_1, \ldots, x_p\}|<p$ using the cyclic (and symmetric) operations of smaller arities, applied to elements that stand out in the repetition pattern in $x_1,  \ldots, x_p$, as follows.

%This occurs when $p\le 5$ in $C_1$ and $C_2$ and $p\le 7$ in $C_3$, for all other cases repetition of elements must necessarily occur.
%Assume now that $|\{x_1, \ldots, x_p\}|<p$. We define

 \end{enumerate}

 We now show that these partial  operations preserve the relations $R$ and $S$.
\begin{claim}
Let $x_1, \ldots, x_p\in K$ be elements from the same partition block.
If $(x_1,y_1), \ldots, (x_p,y_p)\in R$ (respectively $\in S$) then $(c_p(x_1,\ldots, x_p),c_p(y_1,\ldots, y_p))\in R$ (respectively $\in S$).
\end{claim}
\begin{proof}
First note that $y_1, \ldots, y_p$ also belong to the same block.
If $x_1, \ldots, x_p$ are all distinct, then  $c_p(x_1, \ldots, x_p)$ was defined in (1), and
it is not hard to check directly that these partial operations preserve $R$ and $S$.
Note that whenever we have a pattern in the repetition of elements in $x_1, \ldots, x_p$,  this pattern is the same
in $y_1, \ldots, y_p$.
 % preserved when following the arcs of $R$ and $S$, since these relations are permutations of $V$.
  For example  if $x_1=x_2$ and $x_3, \ldots, x_p$ are all distinct  then $y_1=y_2$ and  $y_3, \ldots, y_p$ are  all distinct.
%With this is mind it is easy to check that the operations we are about to define do indeed preserve the edges of both $r$ and $s$.
This immediately implies that the partial operations defined in (2) preserve $R$ and $S$, as a consequence of the partial operations defined in (1) also preserving them.
\end{proof}

We now extend the operations define above to elements not belonging to the same block, at the same time as, by induction on $n$,  defining cyclic operations, $c_n'$, of arity $n$  for all $n< 21$, that preserve the relations in $\kest$.
For $n=2$  we define an idempotent cyclic operation $c_2'$ as follows
$$c_2'(x, y)=\left\{\begin{array}{cl}
x& {\rm if}  \ x\in C_i, y\in C_j , \ i< j\\
y& {\rm if}  \ x\in C_i, y\in C_j , \ i> j\\
c_2(x, y) & {\rm  if} \ x, y \in C_i,
\end{array}\right.$$
with $i,j=1, 2, 3$ and $c_2$ as defined in (1) and (2). It is easy to check that $c_2'$ preserves both $R$ and $S$.

Now, assume that $R$ and $S$ are preserved by cyclic operations, $c_n'$ for all $n<k$. If $k$ is not prime, then $k=mq$ and we know, as in \cite{bartocyclic},  that $c_k'$ can be obtained by composing $c_m'$ and $c_q'$ as follows
$$c_k'(x_1, \ldots, x_k)=c_m'(c_q'(x_1,\ldots, x_q), \ldots, c_q'(x_{k-q+1}, \ldots, x_k)).$$
By the inductive hypothesis,  $c_q'$ and $c_m'$ preserve $R$ and $S$, so  $c_k'$ also preserves these relations.
If $k$ is prime we define
$$c_k'(x_1, \ldots, x_k)=\left\{\begin{array}{cl}
c_k(x_1, \ldots, x_k) & {\rm if}  \ x_1, \ldots , x_k \in C_i,\ \ \  \ (i=1,2,3),\\
c_m'(x_1, \ldots, x_m) & {\rm  if} \ \{x_1, \ldots, x_m\}=  C_1\cap \{x_1, \ldots, x_k\}\ne \emptyset,\\
c_m'(x_1, \ldots, x_m) & {\rm  if} \  \{x_1, \ldots, x_m\}=  C_2\cap \{x_1, \ldots, x_k\} \ {\rm and} \\ & \ \ \ \ C_1\cap \{x_1, \ldots, x_k\}=\emptyset
\end{array}\right.$$
that is:  if all elements $x_1, \ldots, x_k$ belong to the same block then we already know from (1) and (2) that there is a cyclic (partial) operation, $c_k$, defined on them that  preserves $R$ and $S$; if not all elements belong to the same block then we choose the elements in $x_1, \ldots, x_k$ that belong to $C_1$ (or $C_2$ if no element belongs to $C_1$) and apply to them the corresponding operation of smaller arity, which we know exists by the inductive hypothesis. Since the blocks are disjoint and have no arcs connecting them, $c_k'$ clearly preserves both $R$ and $S$. Theorem~\ref{cic} is proved.

\section{Conclusion}\label{K}

We have described the forbidden structures for the existence of symmetric term operations in a finite algebra. We have also shown that the classes of finite algebras
having cyclic operations of all arities and symmetric operations of all arities are not the same. In fact, the algebra
$\algA_{\kest}$ that separates these classes can easily be shown to generate an arithmetical variety.

It is an interesting open question whether Theorem~\ref{2cycles} can be strengthened by requiring the algebra
$\algB$ in $\var(\algA)$ (that has two automorphisms without a common fixed point) to belong to $HS(\algA)$.
This strengthening could help in the study of complexity (more specifically, robust algorithms) for constraint
satisfaction problems~\cite{dalmau:robust}. Even obtaining an upper bound on the number $n$ such that $\algB$ can be found in
$HS(\algA^n)$ would be interesting, since such a bound would imply decidability of the existence of symmetric term operations in a finite algebra
(and hence decidability of the problem of recognising CSPs solvable by linear programming), which is currently an open question.

\end{document}

%% file: defpart.tex
\newcommand{\var}{\operatorname{var}}

%% file: definitions.tex
\long\def\BEGINCOMMENT #1\ENDCOMMENT{\relax}

\newcommand{\maps}\longrightarrow
\newcommand{\cmaps}\Longrightarrow
\newcommand{\vocab}{\tau}

\newcommand{\best}{\operatorname{{\bf B}}}
\newcommand{\aest}{\operatorname{{\bf A}}}
\newcommand{\kest}{\operatorname{{\bf K}}}